\documentclass{article}

\usepackage[left=1.3in,top=1.3in,right=1.3in,bottom=1.3in,nohead]{geometry}

\usepackage{url,subfig,latexsym,amsmath,amssymb, amsfonts,enumerate,
  epsfig, graphics,times, amsthm}
  
 \usepackage[colorlinks]{hyperref}

\newtheorem{theorem}{Theorem}
\theoremstyle{definition}
\newtheorem{example}[theorem]{Example}

\newtheorem{definition}[theorem]{Definition}

\newcommand{\eqdef}{\overset{\mbox{\tiny def}}{=}}
\newcommand{\R}{\mathbb{R}}
\newcommand{\Z}{\mathbb{Z}}
\newcommand{\V}{\mathcal V}
\def\S{\mathcal S}
\def\C{\mathcal C}
\newcommand{\Reac}{\mathcal R}
\newcommand{\po}{\text{Po}}

\title{Lyapunov functions, stationary distributions, and non-equilibrium potential  for  reaction networks}

\author{David F. Anderson\footnote{Department of Mathematics, University of Wisconsin-Madison; {\tt anderson@math.wisc.edu}}, \ \ Gheorghe Craciun\footnote{Department of Mathematics and Department of Biomolecular Chemistry, University of Wisconsin-Madison; {\tt craciun@math.wisc.edu}}, \ \ Manoj Gopalkrishnan\footnote{School of Technology and Computer Science, Tata Institute of Fundamental Research, Mumbai, India; {\tt manojg@tifr.res.in}}, \ \ Carsten Wiuf\footnote{Department of Mathematical Sciences, University of Copenhagen; {\tt wiuf@math.ku.dk}}}

\begin{document}

\maketitle

\abstract{
	We consider the relationship between stationary distributions for stochastic models of  reaction systems and Lyapunov functions for their deterministic counterparts.  Specifically, we derive the well known Lyapunov function of reaction network theory as a scaling limit of the non-equilibrium potential  of the stationary distribution of stochastically modeled complex balanced systems. We extend this result to general birth-death models and demonstrate via example that similar scaling limits can yield Lyapunov functions even for models that are not complex or detailed balanced, and may even have multiple equilibria.
}

\section{Introduction}

Reaction network models are ubiquitous in the study of various types of population dynamics in biology.
For example, they are used in modeling subcellular processes in molecular biology \cite{AK2015,Elowitz2002,Gunawardena2005,ShF2010}, signaling systems \cite{Sontag2001,TCN2003}, metabolism \cite{DRN2013}, as well as the spread of infectious diseases \cite{Allen2003} and interactions between species in an ecosystem \cite{May2001,SmithTh2011}.    Depending upon the relevant scales of the system, either a deterministic or stochastic model of the dynamics is utilized.

This paper studies the connection between deterministic and stochastic models of reaction systems.  In particular, for the class of so-called ``complex balanced'' models, we make a connection between the stationary distribution of the stochastic model and the classical Lyapunov function used in the study of the corresponding deterministic models.  Specifically, we show that in the  large volume limit of Kurtz \cite{Kurtz72,Kurtz78}, the non-equilibrium potential  of the stationary distribution of the scaled stochastic model converges to the standard Lyapunov function of deterministic  reaction network theory.   Further, we extend this result to birth-death processes.

In 1972, Horn and Jackson~\cite{HornJack72} introduced a Lyapunov function for the study of complex balanced systems, and remarked on a formal similarity to Helmholtz free energy functions. 
Since then the probabilistic interpretation of this Lyapunov function for complex balanced systems has remained obscure.
For detailed balanced systems, which form a subclass of complex balanced systems, a probabilistic  interpretation for the Lyapunov function is  known --- see, for example, the work of Peter Whittle~\cite[Section~5.8]{Whittle86} --- though  these arguments appear to be little known in the mathematical biology  community. The key ingredient that enables us to extend the analysis pertaining to detailed balanced systems to complex balanced systems comes from~\cite{ACK2010}, where Anderson, Craciun, and Kurtz showed that the stationary distribution for the class of  complex balanced  reaction networks can be represented as a product of Poisson random variables; see equation \eqref{eq:prod_form_dist} below.

While there are myriad results pertaining to either stochastic or deterministic models, there are relatively few making a connection between the two.  Perhaps the best known such connections come from the seminal work of Thomas Kurtz \cite{Kurtz72,Kurtz78,Kurtz80}, which details the limiting behavior of classically scaled stochastic models on \textit{finite} time intervals, and demonstrates the validity of the usual deterministic ODE models on those intervals. There is even less work on the connection between the deterministic and stochastic models on infinite time horizons, that is, on the long term behavior of the different models, though two exceptions stand out.  As alluded to above, Anderson, Craciun, and Kurtz showed that a stochastically modeled complex balanced system --- for which the deterministically modeled system has  complex balanced equilibrium\footnote{By \textit{equilibrium} we mean a fixed point of a dynamical system.  In particular, what is referred to in the biochemistry literature as a ``non-equilibrium steady state'' is also included in our use of the term equilibrium.}  $c$ --- has a stationary distribution of product form,
\begin{equation}\label{eq:prod_form_dist}
	\pi(x) = \frac{1}{Z_\Gamma} \prod_{i=1}^d \frac{ c_i^{x_i}}{x_i!}, \quad x \in \Gamma \subset \Z^d_{\ge 0},
\end{equation}
where $\Gamma$ is the state space of the stochastic model and $Z_\Gamma>0$ is a normalizing constant \cite{ACK2010}.  On the other hand, in \cite{AndACRStoch2014}, Anderson, Enciso, and Johnston provided a large class of networks for which the limiting behaviors of the stochastic and deterministic models are fundamentally different, in that the deterministic model has special ``absolutely robust'' equilibria whereas the stochastic model necessarily undergoes an extinction event.  

In the present paper, we return to the context of complex balanced models studied in \cite{ACK2010}, and show that the usual Lyapunov function of Chemical Reaction Network Theory (CRNT),
\begin{equation}\label{eq:standardLyapunov}
	\V(x) = \sum_i x_i\left( \ln(x_i) - \ln( c_i) -1\right)  +  c_i,
\end{equation}
can be understood as the limit of the non-equilibrium potential  of the distribution \eqref{eq:prod_form_dist} in the classical scaling of Kurtz. We extend this result to the class of birth-death models.  We then demonstrate through examples that Lyapunov functions for an even wider class of models can be constructed through a similar scaling of stationary distributions. It is not yet clear just how wide the class of models for which this specific scaling limit provides a Lyapunov function is, and we leave this question open.  Similar (non-mathematically rigorous) results have been pointed out in the physics literature though the generality of these results remain unclear \cite{Qian11}.  
See also \cite{GK2013} for recent mathematical work pertaining to the ergodicity of stochastically modeled reaction systems and \cite{PCK2014} for earlier related work pertaining to the irreducibility and recurrence properties of stochastic models.

Before proceeding, we provide a key definition.
\begin{definition}
	 Let $\pi$ be a probability distribution on a countable set $\Gamma$ such that $\pi(x)>0$ for all $x \in \Gamma$.  The \textit{non-equilibrium potential} of the distribution $\pi$ is the function $\phi_\pi\colon \Gamma \to \R$ defined by
	\[
		\phi_\pi(x) =- \ln\left(\pi(x)\right).
	\]
\end{definition}

We close the introduction with  an illustrative example.

\begin{example}\label{ex:30984}
	Consider the catalytic activation-inactivation network
	\begin{align}\label{ex:first}	
		2A \rightleftharpoons A+B, 
	\end{align}	
	where $A$ and $B$ represent the active and inactive forms of a protein, respectively.  The usual {\em deterministic} mass-action kinetics model for the concentrations $(x_A,x_B)$ of the species $A$ and $B$ is
	\begin{align*}
		\dot x_A &= -\kappa_1 x_A^2 + \kappa_2 x_Ax_B\\
		\dot x_B &=\phantom{-} \kappa_1 x_A^2 - \kappa_2 x_Ax_B,
	\end{align*}
	where $\kappa_1$ and $\kappa_2$ are the corresponding reaction rate constants for the forward and reverse reactions in \eqref{ex:first}.   For a given total concentration $M \eqdef x_A(0) + x_B(0) >0$, these equations have a unique stable  equilibrium 
	\begin{align}\label{eq:eq-example}
			 c_A = \frac{M\kappa_2}{\kappa_1+\kappa_2},\qquad
			 c_B =\frac{M\kappa_1}{\kappa_1+\kappa_2},
	\end{align}
	which can be shown to be complex balanced.
	
	We now turn to a {\em stochastic} model for the network depicted in \eqref{ex:first}, that tracks the molecular counts for species $A$ and $B$. Letting $V$ be a scaling parameter, which can be thought of as Avogadro's number multiplied by volume, see section \ref{sec:scaling}, the standard stochastic mass-action kinetics model can be described in several different ways.  
	For example, the Kolmogorov forward equations governing the probability distribution of the process are
	\begin{align}
	\label{eq:forward_first}
	\begin{split}
		\frac{d}{dt} p_{\mu}(x_A,x_B,t) &= \frac{\kappa_1}{V} (x_A+1)x_Ap_\mu(x_A+1,x_B-1,t) \\
		&\hspace{.2in} + \frac{\kappa_2}{V} (x_A-1)(x_B+1)p_\mu(x_A-1,x_B+1,t)\\
		&\hspace{.2in} - \left[\frac{\kappa_1}{V} x_A(x_A-1) + \frac{\kappa_2}{V} x_Ax_B \right] p_\mu(x_A,x_B,t),
		\end{split}
	\end{align}
	where $x_A,x_B \in \Z_{\ge 0}$ are the molecular counts of $A$ and $B$, respectively, and $p_\mu(x_A,x_B,t)$ denotes the probability that the system is in state $(x_A,x_B)$ at time $t$ given an initial distribution of $\mu$.  Note that there is one such differential equation for each state, $(x_A,x_B)$,  in the state space.  In the biological context the forward equation is typically referred to as the \emph{chemical master equation}.   
	 
	 Assume that the initial distribution for the stochastic model has support on the set $\Gamma^V \eqdef \{(x_A,x_B) \in \Z^2_{\ge 0} | x_A\ge 1, x_A + x_B = VM\}$, where $M>0$ is fixed and $V$ is selected so that $VM$ is an integer.  Hence, the total number of molecules is taken to scale in $V$. The stationary distribution  can then be found by setting the left hand side of the forward equation \eqref{eq:forward_first} to zero and solving the resulting  system of equations (one equation for each $(x_A,x_B) \in \Gamma^V$).  Finding such a solution is typically a challenging, or even impossible task.  However, results in \cite{ACK2010} imply that for this particular system the stationary distribution is (almost) a binomial distribution and is of the form \eqref{eq:prod_form_dist},
  	 \begin{equation}\label{eq:948787}
	  	\pi^V\!(x_A,x_B) = \frac{1}{Z^V} \binom{VM }{x_A} \left( \frac{\kappa_2}{\kappa_1+\kappa_2}\right)^{\!x_A} \left(\frac{\kappa_1}{\kappa_1+\kappa_2}\right)^{\!x_B}, \quad (x_A,x_B) \in \Gamma^V,
	\end{equation}
where $Z^V$ is the normalizing constant
\[
	Z^V\eqdef  1-\left(\frac{\kappa_1}{\kappa_1+\kappa_2}\right)^{\!\!VM}.
\]
 The distribution is not binomial since the state $(x_A,x_B)=(0,VM)$ cannot be realized in the system.

In order to make a connection between the stochastic and deterministic models, we  convert the stochastic model to concentrations by dividing by $V$.  
 That is, for $x\in \Z$ we let $\tilde x^V\eqdef V^{-1}x$.  Letting $\tilde \pi^V\!(\tilde x^V)$ denote the stationary distribution of the scaled process, we find that 
 \[	\tilde \pi^V\!(\tilde x^V )  = \frac{1}{Z^V} \binom{VM}{V \tilde x^V_A} \left( \frac{\kappa_2}{\kappa_1+\kappa_2}\right)^{\!V\tilde x^V_A} \left(\frac{\kappa_1}{\kappa_1+\kappa_2}\right)^{\!V\tilde x^V_B},\]
where $\tilde x^V \in \frac{1}{V} \Gamma^V$. We now consider the non-equilibrium potential  of $\tilde \pi^V$ scaled by $V$
 \begin{align*}
	-\frac{1}{V}\ln (\tilde \pi^V\!(\tilde x^V)) &=\frac{1}{V}\ln( Z^V)- \frac{1}{V}\ln( (VM)!) + \frac{1}{V}\ln((V\tilde x_A^V)!) + \frac{1}{V}\ln((V\tilde x^V_B)!)\\
	&\hspace{.2in} - \tilde x^V_A \ln\left( \frac{\kappa_2}{\kappa_1+\kappa_2}\right) - \tilde x^V_B\ln\left( \frac{\kappa_1}{\kappa_1+\kappa_2}\right).
 \end{align*}
 Stirling's formula says that 
	 \begin{equation}\label{eq:249850}
	 	\ln(n!) = n\ln(n) - n + O(\ln(n))\quad \text{for} \quad n>0.
	\end{equation}  
Assuming that $\lim_{V \to \infty} \tilde x^V= \tilde x \in \R^2_{> 0}$, and after some calculations, equation \eqref{eq:249850} yields	
	 \begin{align*}
		\lim_{V\to \infty} -\frac{1}{V}\ln (\tilde \pi^V\!(\tilde x^V)) &= \tilde x_A\left( \ln \tilde x_A - \ln \left( \frac{\kappa_2}{\kappa_1 + \kappa_2}\right)  \right) \\
		&\hspace{.2in} + \tilde x_B\left( \ln(\tilde x_B) - \ln\left( \frac{\kappa_1}{\kappa_1 + \kappa_2}\right)  \right) - M\ln(M)\\
		&\eqdef \V(\tilde x).
 \end{align*}
  Recalling that $\tilde x_B = M - \tilde x_A$, we may rewrite $\V$ in the following useful way
	 \begin{align*}
	 	\V(\tilde x) &= \tilde x_A \left( \ln \tilde x_A - \ln \left( \frac{M \kappa_2}{\kappa_1 + \kappa_2}\right)-1  \right)  -\frac{M\kappa_2}{\kappa_1 + \kappa_2}\\
		&\hspace{.2in}+\tilde x_B \left( \ln\tilde  x_B - \ln \left( \frac{M\kappa_1}{\kappa_1 + \kappa_2}\right)-1  \right)  -\frac{M\kappa_1}{\kappa_1 + \kappa_2}.
	 \end{align*}
Remarkably, this $\V(\tilde x)$ is exactly the function we would obtain if we were to write the standard Lyapunov function of CRNT, given in  \eqref{eq:standardLyapunov}, for this model. \hfill $\square$
\end{example}

\vspace{.2in}

The first  goal of this paper is to show that the equality between the \textit{scaling  limit} calculated for the stochastic model above, and the Lyapunov function for the corresponding deterministic model is not an accident, but in fact holds for all complex balanced systems.  We will also demonstrate that  the  correspondence holds for a wider class of models.  

The remainder of this article is organized as follows.   In Section \ref{sec:background}, we briefly review some relevant terminology and results.  In Section \ref{sec:scalings}, we derive the general Lyapunov function of Chemical Reaction Network Theory for complex balanced systems as a scaling limit of the non-equilibrium potential  of the corresponding scaled stochastic model.    In Section \ref{sec:example_BD}, we discuss other, non-complex balanced,  models for which the same scaling limit  gives a Lyapunov function for the deterministic model. In particular, we characterize this function when the corresponding stochastic system is equivalent to a stochastic birth-death process.

\section{Reaction systems and previous results}
\label{sec:background}

\subsection{Reaction networks}

We consider a system consisting of  $d$ species, $\{S_1,\dots,S_d\}$,
undergoing transitions due to a  finite number, $m$, of  reactions.  For the $k$th reaction, we denote by $\nu_k, \nu_k' \in \Z^d_{\ge 0}$ the vectors
representing  the number of molecules of each species consumed and
created in one instance of the reaction, respectively.  For example, for the
reaction $S_1 + S_2 \to S_3$, we have $\nu_k = (1,1,0)^T$ and $\nu_k' = (0,0,1)^T$, if there are $d=3$ species in the system. Each $\nu_k$ and $\nu_k'$ is
termed a {\em complex} of the system.  The reaction is  denoted by  $\nu_k \to \nu_k'$, where $\nu_k$ is termed the \textit{source complex} and $\nu_k'$ is the \textit{product complex}.  
A complex may appear as both a source complex and a product complex in
the system. 

\begin{definition}  \label{def:crn}
  Let $\S = \{S_1,\ldots,S_d\}$, $\C = \bigcup_{k=1}^m \{\nu_k,\nu_k'\}$,
  and $\Reac = \{\nu_1 \to  \nu_1', \dots, \nu_{m} \to \nu_{m}'\}$ denote the sets of species, complexes, and reactions,
  respectively.  The triple $\{\S, \C, \Reac\}$ is  a \textit{reaction network}.
\end{definition}

\begin{definition}\label{def:stoich}
  The linear subspace 
 $S = \hbox{span}\{\nu_1' - \nu_1,\dots, \nu_{m}'-\nu_m\}$
  is called the {\em stoichiometric subspace} of the network. For $c \in
  \R^d_{\ge0}$ we say $c + S = \{x\in \R^d | x = c + s \text{ for some } s \in S\}$ is a {\em stoichiometric compatibility class}, $(c+S) \cap \R^d_{\ge0}$ is a {\em non-negative stoichiometric compatibility class}, and $(c+S) \cap \R^d_{>0}$ is a {\em positive stoichiometric compatibility class}. 
\end{definition}

\subsection{Dynamical system models}
\label{sec:dyn}

\subsubsection{Stochastic models} 

The most common stochastic model for a  reaction network   $\{\S, \C, \Reac\}$
treats the system as a continuous time
Markov chain whose state $X$ is a vector giving the number of
molecules of each species present with each reaction modeled as a
possible transition for the chain. The model for the $k$th reaction is determined by the source and product complexes of the reaction,  and a function  $\lambda_k$ of the
state that gives the \textit{transition intensity}, or rate,
at which the reaction occurs. In the biological and chemical literature,
transition intensities are referred to as \textit{propensities}.

Specifically, if the $k$th reaction occurs at time $t$ the state is  updated
by addition of the \textit{reaction vector} $\zeta_k \eqdef \nu_k' - \nu_k$ and
\[X(t) = X(t-) + \zeta_k.\]
The most common choice for intensity functions is to assume the system satisfies the stochastic version of \textit{mass-action kinetics}, which states that the rate functions take the form
\begin{equation}\label{eq:stoch_MA}
	\lambda_k(x) = \kappa_k \prod_{i=1}^d \frac{x_{i}!}{(x_{i}  - \nu_{ki})!}1_{\{x_i \ge \nu_{ki}\}},
\end{equation}
for some constant $\kappa_k>0$,  termed the rate constant, and where $\nu_k=(\nu_{k1},\ldots,\nu_{kd})^T$.
Under the assumption of mass-action kinetics and a non-negative initial condition, it follows that the dynamics of the system is confined to a particular non-negative  stoichiometric compatibility class given by the initial value $X(0)$, namely $X(t)\in (X(0) + S) \cap \R^d_{\ge0}$.

The number of times that the $k$th reaction occurs by time $t$ can be
represented by the counting process
\[R_k(t) = Y_k\left(\int_0^t \lambda_k(X(s))ds\right), \]
 where the $\{Y_k, k \in \{1,\dots,m\}\}$ are independent
unit-rate Poisson processes (see \cite{AndKurtz2011, AK2015, KurtzPop81}, or  \cite[Chapter~6]{Kurtz86}]).  The state of the system then satisfies the  equation $X(t) = X(0) + \sum_k R_k(t)\zeta_k$, or
\begin{align}
  X(t) &=X(0) + \sum_k Y_k\left( \int_0^t \lambda_k(X(s))ds
  \right)\zeta_k
   \label{eq:Rk2},
\end{align}
where the sum is over the reaction channels.
Kolmogorov's forward equation  for this model is
\begin{equation}\label{eq:CME}
  \frac{d}{dt} P_\mu(x, t) = \sum_k \lambda_k(x-\zeta_k) P_\mu(x-
  \zeta_k,t) - \sum_k \lambda_k(x) P_\mu(x,t), 
\end{equation}
where $P_\mu(x,t)$ represents the probability that $X(t) = x\in \Z^d_{\ge 0}$ given an initial distribution of $\mu$ and $\lambda_k(x-\zeta_k) = 0$ if $x-\zeta_k \notin \Z^d_{\ge 0}$.  
So long as the process is non-explosive, the two representations for the processes, the stochastic equation \eqref{eq:Rk2} and the Markov process with forward equation \eqref{eq:CME}, are equivalent \cite{AndKurtz2011,Kurtz86}.

It is of interest to characterize the long-term behavior of the process. Let $\Gamma\subset \Z^d_{\ge 0}$ be a closed component of the state space; that is, $\Gamma$ is closed under the transitions of the Markov chain.
A probability distribution $\pi(x)$, $x\in  \Gamma$, is a stationary distribution for the chain on $\Gamma$ if
\begin{equation}
  \sum_{k} \pi(x - \zeta_k)\lambda_k(x - \zeta_k) =  \pi(x)\sum_k \lambda_k(x)
  \label{eq:stationary2}
\end{equation}
for all $x\in\Gamma$.  (If $x-\zeta_k\not\in\Gamma$ then $\pi(x-\zeta_k)$ is put to zero.) If in addition $\Gamma$ is irreducible, that is,  any state in $\Gamma$ can be reached from any other state in $\Gamma$ (for example, $\Gamma^V$ in Example \ref{ex:30984} is an irreducible component)
 and $\pi$ exists, then $\pi$ is unique \cite{karlin}.

Solving equation \eqref{eq:stationary2} is in general a difficult
task, even when we assume each $\lambda_k$ is determined by mass-action kinetics.
However,  if in addition there  exists a  complex balanced equilibrium for the associated deterministic model, then  equation \eqref{eq:stationary2} can be solved   explicitly, see Theorem \ref{thm:prodform_main} below.

\subsubsection{Deterministic models and complex balanced equilibria}
\label{sec:def0}

For two vectors $u,v \in \R^d_{\ge 0}$ we define $u^v \eqdef \prod_i u_i^{v_i}$ and adopt the convention that $0^0=1$. 

Under an appropriate scaling limit (see Section \ref{sec:scaling})
the continuous time Markov chain model described in the previous section becomes
\begin{equation}
  x(t) = x(0) + \sum_{k} \left(\int_0^t f_k(x(s)) ds \right) (\nu_k' -
  \nu_k), 
  \label{eq:cont}
\end{equation}
where
\begin{equation}
  f_{k}(x) = \ \kappa_{k}x_1^{\nu_{k1}} x_2^{\nu_{k2}} \cdots x_d^{\nu_{kd}} =\kappa_k x^{\nu_k},
  \label{eq:massaction}
\end{equation}
and  $\kappa_k>0$ is a constant.  We say that the deterministic system \eqref{eq:cont} has {\em deterministic mass-action kinetics} if the rate functions $f_k$ have the form \eqref{eq:massaction}.  The system \eqref{eq:cont} is equivalent to the system of {\em ordinary differential equations} (ODEs) with a given initial condition $x_0=x(0)$,

\begin{equation}\label{eq:mass-action1}
   \dot{x} = \sum_k \kappa_k  x^{\nu_{k}}(\nu_k' - \nu_k).
\end{equation}
The trajectory with initial condition $x_0$ is confined to the non-negative stoichiometric compatibility class $(x_0+S)\cap \R^d_{\ge 0}$.

  Some mass-action systems have  complex balanced equilibria \cite{Horn72,HornJack72},\footnote{For example, it is known that all weakly reversible networks with a deficiency of zero give rise to systems that have complex balanced equilibria \cite{FeinbergLec79,Feinberg95a}.}  which have been shown to play an important role in many biological mechanisms \cite{CLWBNE2012,Gille2009,kang2012new,Sontag2001}.    An equilibrium point $c\in\R^d_{\ge 0}$ is said to be complex balanced if and only if for each complex $z \in \C$ we have
\begin{equation}
  \sum_{\{k : \nu_k' = z \}} \kappa_k c^{\nu_k} = \sum_{ \{ k:\nu_k=z \} } \kappa_k c^{\nu_k}, 
 \label{eq:complex_balanced}
\end{equation}
where the sum on the left is over reactions for which $z$ is the
product complex and the sum on the right is over reactions for which
$z$ is the source complex. For such an equilibrium the total inflows and the total outflows balance out at each complex   \cite{FeinbergLec79,Gun2003}.

In \cite{HornJack72} it is shown that if there exists a complex balanced 
  equilibrium $c \in \R^d_{> 0}$ for a given model then
  \begin{itemize}
  \item[(1)]   There is one, and only one, positive equilibrium point in each  positive stoichiometric compatibility class.
    \item[(2)] Each such equilibrium point is complex balanced.
  \item[(3)] Each such complex balanced equilibrium point  is locally asymptotically stable relative to  its stoichiometric compatibility class.
    \end{itemize}
    
     Whether or not each complex balanced equilibrium is globally asymptotically stable relative to its positive stoichiometric compatibility class is the content of the Global Attractor Conjecture, which has received considerable attention \cite{AndGAC_ONE2011,AndGlobal,AndShiu,CNP2013,GMS2014,Pantea2012}.  The local asymptotic stability is concluded by an application of the Lyapunov function \eqref{eq:standardLyapunov}.

\subsubsection{Lyapunov functions}

\begin{definition}\label{def:lya}
Let $E \subset \R^d_{\ge 0}$ be an open subset of $\R^d_{\ge 0}$  and let $f: \R^d_{\ge 0} \to \R$.  A  function $\V\colon E \to \R$  is called a (strict) \emph{Lyapunov function} for the system $\dot{x}=f(x)$ at $x_0\in E$ if $x_0$ is an equilibrium point for $f$,  that is, $f(x_0)=0$, and 
\begin{itemize}
\item[(1)] $\V(x)>0$ for all $x\not= x_0$, $x\in E$ and $V(x_0)=0$
\item[(2)] $\nabla{\V}(x)\cdot f(x)\le 0$, for all $x\in E$, with equality if and only if $x=x_0$, where $\nabla{\V}$ denotes the gradient of $\V$.
\end{itemize}
\end{definition}

If these two conditions are fulfilled then the equilibrium point $x_0$ is  {\em asymptotically stable}  \cite{perko}. If the inequality in (2) is not strict for $x_0\not=x$ then $x_0$ is {\em stable} and not necessarily asymptotically stable. If the inequality is reversed, $\dot{\V}(x)>0$, $x\not=x_0$, then the equilibrium point is {\em unstable} \cite{perko}.

We will see that in many cases the large volume limit of the non-equilibrium potential of a stochastically modeled system is a Lyapunov function defined on the interior of the nonnegative stoichiometric subspace.

\subsection{Product form stationary  distributions}
\label{sec:prod_form}

The following result from \cite{ACK2010},  utilized in \eqref{eq:948787}, provides a characterization of the stationary distributions of complex balanced systems.  See also \cite{Othmer2005,HQ2006} for related work.

\begin{theorem}  \label{thm:prodform_main}
  Let $\{\S, \C, \Reac\}$ be a reaction network and let  $\{\kappa_k\}$ be a choice of rate constants.  Suppose that, modeled deterministically,  the system is complex balanced with a complex  balanced equilibrium $c \in \R^d_{>0}$.  Then the stochastically  modeled system with intensities \eqref{eq:stoch_MA} has a stationary distribution on $\Z^d_{\ge 0}$ consisting of the product of Poisson distributions,
    \begin{equation}
    \pi(x) = \prod_{i=1}^d \frac{c_i^{x_i}}{x_i!}e^{-c_i}, \qquad \text{ for } x \in
    \Z^d_{\ge 0}.
    \label{eq:stationary_product}
  \end{equation}
  If $\Z^d_{\ge 0}$ is irreducible, then \eqref{eq:stationary_product}
  is the unique stationary distribution.   
   If $\Z^d_{\ge 0}$ is not irreducible, then the stationary distribution, $\pi_\Gamma$, of an irreducible component of the state space $\Gamma\subset \Z^d_{\ge 0}$ is 
  \begin{equation*}
    \pi_{\Gamma}(x) = \frac{1}{Z_{\Gamma}} \prod_{i = 1}^d
    \frac{c_i^{x_i}}{x_i!}e^{-c_i}, \qquad \text{ for }  x \in \Gamma,
  \end{equation*}
  and $\pi_{\Gamma}(x) = 0$ otherwise, where $Z_{\Gamma}$ is a
  positive normalizing constant. 
\end{theorem}

Each irreducible component of the state space is necessarily contained in a single non-negative stoichiometric compatibility class (Definition \ref{def:stoich}). The choice of the complex balanced equilibrium point $c$ in the theorem is independent of  $\Gamma$ and the particular stoichiometric compatibility class containing it \cite{ACK2010}.
Since  $\Gamma\subset \Z^d_{\ge 0}$, it follows that
\begin{equation}\label{eq:ZV}
Z_\Gamma = \sum_{x\in \Gamma} \prod_{i = 1}^d  \frac{c_i^{x_i}}{x_i!}e^{-c_i} \le \sum_{x \in \Z^d_{\ge 0}}\prod_{i = 1}^d    \frac{c_i^{x_i}}{x_i!}e^{-c_i} = 1.
\end{equation}

\subsubsection{The classical scaling}
\label{sec:scaling}

We may convert from molecular counts to concentrations by scaling the counts by $V$, where $V$ is the volume  of the system times Avogadro's number.  Following \cite{ACK2010},  define $|\nu_k| =\sum_i \nu_{ki}$. Let $\{\kappa_k\}$ be a set of rate constants and define the  scaled rate constants, $ \kappa^V_k$, for the stochastic model in the following way,
\begin{equation}
 \kappa^V_k = \frac{\kappa_k}{V^{|\nu_k|-1}}
 \label{eq:kappas}
 \end{equation}
  (see \cite[Chapter~6]{Wilkinson2006}).   Let $x\in \Z^d_{\ge 0}$ be an arbitrary state of the system and denote the intensity function for the stochastic model by 
 \begin{equation*}
 	\lambda_k^V(x) = \frac{V \kappa_k }{V^{|\nu_k|}} \prod_{i = 1}^d \frac{x_i!}{(x_i-\nu_{ki})!}.
\end{equation*}
  Note that $\tilde x \eqdef V^{-1}x$ gives the concentrations in moles per
unit volume and that if $\tilde x = \Theta(1)$ (that is, if $x = \Theta(V)$), then by standard arguments
\begin{equation*}
	\lambda_k^V(x) \approx V \kappa_k \prod_{i=1}^d \tilde x_i^{\nu_{ki}} = V \lambda_k(\tilde x),
\end{equation*}
where the final equality determines $\lambda_k$, and justifies the definition of deterministic mass-action kinetics in \eqref{eq:massaction}.  

Denote the stochastic process determining the abundances  by $X^V\!(t)$ (see \eqref{eq:Rk2}).
Then, normalizing the  original process $X^V$ by $V$ and defining $\widetilde X^V \eqdef V^{-1}X^V$ yields 
\begin{equation*}
	\widetilde X^V(t) \approx \widetilde X^V(0) + \sum_k \frac{1}{V} Y_k\left( V \int_0^t  \lambda_k(\widetilde X^V(s)) ds \right)\zeta_k.
\end{equation*}
 Since the law of large numbers for the Poisson
process implies $V^{-1}Y(Vu)\approx u$, we may conclude that a good approximation to the process $\widetilde X^V$ is the function $x=x(t)$ defined as the solution to the ODE
\begin{equation*}
   \dot{x} = \sum_k \kappa_k  x^{\nu_{k}}(\nu_k' - \nu_k),
\end{equation*}
which is \eqref{eq:mass-action1}.
For a precise formulation of the above scaling argument, termed the \textit{classical scaling}, see \cite{AK2015}.

The following is an immediate corollary to Theorem \ref{thm:prodform_main}, and can also be found in \cite{ACK2010}.  The result  rests upon the fact that if $c$ is a complex balanced equilibrium for a given reaction network with rates $\{\kappa_k\},$ then $Vc$ is a complex balanced equilibrium  for the reaction network endowed with rates $\{\kappa^V_k\}$ of \eqref{eq:kappas}.

\begin{theorem}
Let $\{\S, \C, \Reac\}$ be a reaction network and let  $\{\kappa_k\}$ be a choice of rate constants.  Suppose that, modeled deterministically,  the system is complex balanced with a complex  balanced equilibrium $c \in \R^d_{>0}$.  For some $V>0$, let $\{\kappa^V_k\}$ be related to $\{ \kappa_k\}$ via \eqref{eq:kappas}.  Then the stochastically
  modeled system with intensities \eqref{eq:stoch_MA} and rate constants $\{\kappa^V_k\}$ has a stationary distribution on $\Z^d_{\ge 0}$ consisting of the product of Poisson distributions,
  \begin{equation}
    \pi^V(x) = \prod_{i=1}^d \frac{(Vc_i)^{x_i}}{x_i!}e^{-Vc_i}, \qquad \text{ for }  x \in
    \Z^d_{\ge 0}.
    \label{eq:309847}
  \end{equation}
  If $\Z^d_{\ge 0}$ is irreducible, then \eqref{eq:309847}
  is the unique stationary distribution.  
   If $\Z^d_{\ge 0}$ is not irreducible, then the stationary distribution, $\pi_\Gamma^V$, of an irreducible component of the state space $\Gamma\subset \Z^d_{\ge 0}$ is
  \begin{equation}
  \label{eq:89767896}
    \pi^V_{\Gamma}(x) = \frac{1}{Z_\Gamma^V} \prod_{i = 1}^d
    \frac{(Vc_i)^{x_i}}{x_i!} e^{-Vc_i}, \qquad \text{ for } x \in \Gamma,
  \end{equation}
  and $\pi^V_{\Gamma}(x) = 0$ otherwise, where $Z^V_{\Gamma}$ is a
  positive normalizing constant.
  \label{thm:rate_constants}
\end{theorem}

Note that Theorem \ref{thm:rate_constants} implies that a stationary distribution for the scaled model $\widetilde X^V$  is
\begin{equation}\label{eq:4309}
\tilde \pi^V\!(\tilde x^V) = \pi^V\!(V \tilde x^V), \quad \text{ for }\quad \tilde x^V \in \frac{1}{V}\Z^d_{\ge 0} .
\end{equation}

\section{Complex balanced systems}
\label{sec:scalings}

We are ready to state and prove our first  result.  For an increasing series of volumes $V_i$, $i=1,2,\ldots$, we consider converging sequences of points $\tilde x^{V_i}$ in $\frac{1}{V_i} \Z^d_{\ge 0}$. To ease the notation we omit the index $i$ and write, for example, $\lim_{V\to \infty} \tilde x^V$ instead of $\lim_{i \to \infty} \tilde x^{V_i}$.  

\begin{theorem}\label{comp-main}
Let $\{\S, \C, \Reac\}$ be a  reaction network and let  $\{\kappa_k\}$ be a choice of rate constants.  Suppose that, modeled deterministically,  the system is complex balanced.  For  $V>0$, let $\{\kappa^V_k\}$ be related to $\{ \kappa_k\}$ via \eqref{eq:kappas}. Fix a sequence of points  $\tilde x^V\in \frac1V \Z^d_{\ge 0}$  for which  $\lim_{V\to \infty}\tilde {x}^V = \tilde x \in \R^d_{>0}$.  Further let $c$ be the unique complex balanced equilibrium within the positive stoichiometric compatibility class of $\tilde x$. 
  
    Let $\pi^V$ be given by \eqref{eq:309847} and let $\tilde \pi^V$ be as in \eqref{eq:4309}, then
  \begin{equation*}
  	\lim_{V\to \infty}\left[- \frac{1}{V}\ln(\tilde \pi^V\!(\tilde x^V)) \right]= \V(\tilde x),
  \end{equation*}
  where $\V$ satisfies \eqref{eq:standardLyapunov}. In particular, $\V$ is a Lyapunov  function (Definition \ref{def:lya}).
 
  Further, suppose  $\Gamma^V\!\subset Z^d_{\ge 0}$ is an irreducible component of the state space for the Markov model with rate constants $\{\kappa_k^V\}$ such that $V \cdot \tilde x^V \in \Gamma^V$. Let $\pi^V_{\Gamma^V}$ be given by \eqref{eq:89767896}.  For $\tilde w^V \in \frac{1}{V}\Gamma^V$, define $\tilde \pi^V_{\Gamma^V}(\tilde w^V) \eqdef \pi^V_{\Gamma^V}(V\tilde w^V)$, then  
  \begin{align}\label{eq:78969786}
  	\lim_{V\to \infty} V^{-1} \ln(Z^V_{\Gamma^V}) = 0,
  \end{align}
  and
  \begin{equation}\label{eq:456789}
  	\lim_{V\to \infty} \left[ - V^{-1}\ln(\tilde \pi_{\Gamma^V}^V\!(\tilde x^V)) \right]= \V(\tilde x),
  \end{equation}
  where $\V$ satisfies \eqref{eq:standardLyapunov}. In particular, $\V$ is a Lyapunov  function (Definition \ref{def:lya}).
\label{thm:main}
\end{theorem}

\begin{proof}
We prove the second statement.  The proof of the first is the same with the exception that $Z_{\Gamma^V}^V \equiv 1$.

We  first consider the limit \eqref{eq:78969786}.  Begin by supposing that there is a sequence $\tilde y^V\in\frac{1}{V}\Gamma^V$ for which $\tilde y^V\to c$.  In this case, 
\[
	1 \ge Z_{\Gamma^V}^V= \sum_{y\in\Gamma^V}\prod_{i = 1}^d\frac{(Vc_i)^{y_i}}{y_i!} e^{-Vc_i} \ge \prod_{i = 1}^d\frac{(Vc_i)^{V\tilde y^V_i}}{(V\tilde y^V_i)!} e^{-Vc_i}\ge C\prod_{i = 1}^d\frac{1}{\sqrt{V \tilde y_i^V}}\left(\!\frac{c_i}{\tilde y_i^V}\!\right)^{\!V \tilde y_i^V}\!\! e^{V(\tilde y^V_i-c_i)},
\]
where the first inequality follows from \eqref{eq:ZV}  and the third from an application of Stirling's formula ($C$ is a constant). 
Taking the logarithm and dividing by $V$, it follows that $\lim_{V\to \infty} V^{-1} \ln(Z_{\Gamma^V}^V) = 0$.
Thus, the limit \eqref{eq:78969786} will be shown so long as we can prove the existence of the sequence $\tilde y^V\in \frac{1}{V}\Gamma^V$ converging to the complex balanced equilibrium $c$.

For $M>0$, define the set $(M + \Z^d_{\ge 0}) = \{ w \in \Z^d : w_i \ge M \text{ for each } i \in \{1,\dots, d\}\}$.  From the remark below  Lemma 4.1 in \cite{PCK2014}, there is an $M_0>0$ so that for all $V$ large enough
\begin{equation}\label{eq:67896798}
	\Gamma^V \cap (M_0 + \Z^d_{\ge 0}) = \big(V\cdot \tilde x^V + \text{span}_{\Z}\{\zeta_k\}\big) \cap (M_0 + \Z^d_{\ge 0}).
\end{equation}
Thus, for $V$ large enough, $\Gamma^V$ has constant positive density on its stoichiometric compatibility class.  Let $V \cdot \tilde c^V$ be the unique complex balanced equilibrium in the positive stoichiometric compatibility class of $V\cdot \tilde x^V$.  It follows that $\tilde c^V$ is the unique complex balanced equilibrium in the positive stoichiometric compatibility class of $\tilde x^V$, from which we may conclude that $\lim_{V\to \infty}(\tilde c^V - c)=0$ (since $\tilde x^V \to \tilde x$) \cite{CraciunShiu09}.  Finally, define $\tilde y^V$ via the relation 
\[
	V \cdot \tilde y^V =  [ V \tilde c^V ],
\]
where $[V\tilde c^V]$ is a minimizer of $f(z) = |z - V\tilde c^V|$ over the set  $\Gamma^V \cap (M_0 + \Z^d_{\ge 0})$.  Note that $\tilde y^V \in \frac1V\Gamma^V$.  From \eqref{eq:67896798}, we see that $\tilde y^V - \tilde c^V = O(V^{-1})$, which, when combined with  $\lim_{V\to \infty}(\tilde c^V - c) \to 0$, gives the desired result.

We now turn to \eqref{eq:456789}.  We have 
\begin{align*}
	-V^{-1}\ln\big(&\tilde \pi_{\Gamma^V}^V(\tilde x^V)\big) = - V^{-1}\ln\left( \prod_{i = 1}^d e^{-V  c_i } \frac{(V  c_i)^{V \tilde x_i^V}}{(V \tilde x_i^V)!}\right) + V^{-1} \ln(Z^V_{\Gamma^V})\\
	&=-V^{-1} \sum_{i=1}^d \left[ -V c_i +  (V\tilde x_i^V)\ln(V) + (V\tilde x_i^V) \ln( c_i)  - \ln\left((V \tilde x_i^V)!\right) \right] + V^{-1} \ln(Z^V_{\Gamma^V}).
\end{align*}
Applying Stirling's formula \eqref{eq:249850} to the final term and performing some algebra yields
\begin{align*}
	-V^{-1}\ln(\tilde \pi_{\Gamma^V}^V(\tilde x^V)) &= -V^{-1} \sum_{i=1}^d \left\{ -V c_i +  (V\tilde x^V_i)\ln(V) + (V\tilde x^V_i) \ln( c_i)\right.  \\
	&\hspace{.2in}- \left. \left[ (V\tilde x^V_i) \ln(V\tilde x^V_i) - (V\tilde x^V_i) + O(\ln(V\tilde x^V_i))\right]\right\} + V^{-1} \ln(Z^V_{\Gamma^V})\\[1ex]
	&= \sum_{i=1}^d \left[\tilde x^V_i \{\ln(\tilde x^V_i) - \ln( c_i) - 1\} +  c_i\right] + O(V^{-1}\ln(V\tilde x_i^V))  + V^{-1} \ln(Z^V_{\Gamma^V}).
\end{align*}
The sum is the usual Lyapunov function $\V$, and the result is shown after letting $V\to \infty$, utilizing \eqref{eq:78969786}, and recalling that $\tilde x^V \to \tilde x \in \R^d_{>0}$.
\end{proof}

The theorem above can  be applied to Example \ref{ex:30984}. The unique equilibrium point given in \eqref{eq:eq-example} is easily seen to fulfil the complex balanced condition in \eqref{eq:complex_balanced}.

\section{Non-complex balanced systems}
\label{sec:example_BD}

\subsection{Birth-death processes and reaction networks}

In this section we will study reaction networks that also are birth-death processes. Many results are known for birth-death processes.  In particular, a characterization of the stationary distribution can be accomplished \cite{karlin}.

Let $\{\S, \C, \Reac\}$ be a  reaction network with one species only, $\S=\{S\}$, and assume all reaction vectors are either $\zeta_k=(-1)$ or $\zeta_k=(1)$. This implies that the number of molecules of $S$ goes up or down by one each time a reaction occurs. 
 For convenience, we re-index the reactions and the reaction rates in the following way.  By assumption, a reaction of the form $n S\to n'S$ will either have $n'=n+1$ or $n'=n-1$. In the former case we index the reaction by $n$ and denote the rate constant by $\kappa_n$ and in the latter case by $-n$ and $\kappa_{-n}$, respectively.  Note  that this stochastically modeled reaction network  can be considered as a birth-death process with birth and death rates
  \begin{align}\label{bd}
  \begin{split}
p_i &= \sum_{\{n\vert \zeta_n=(1)\}}  \lambda^V_n(i) = \sum_{\{n\ge 0\}}  \lambda^V_n(i), \\
q_i &= \sum_{\{n\vert \zeta_n=(-1)\}}    \lambda^V_n(i)= \sum_{\{n<0\}}  \lambda^V_n(i),
\end{split}
\end{align}
for $i\ge 0$, respectively. 

If the stochastically modeled system has absorbing  states (i.e.~states for which $p_i=q_i=0$) we make the following modification to the intensity functions of the system. Let $i_0\in\Z_{\ge 0}$ be the smallest value such that (i) all birth rates of $i_0$ are non-zero, that is, $\lambda_n(i_0)>0$ for $n\ge 0$, and (ii) all death rates of $i_0+1$ are non-zero, that is, $\lambda_n(i_0+1)>0$ for $n<0$.  We modify the system by letting $\lambda_n(i_0)=0$ for $n<0$. Note that the modified system has a lowest state $i_0$, which is not absorbing.  

As an example of the above modification, consider the system with network
\begin{align}
	3S \overset{\kappa_{-3}}{\to} 2S, \qquad 4S \overset{\kappa_{4}}{\to} 5S.
\label{eq:98767896}
	\end{align}
This model has rates $\lambda_4(x) = \kappa_4 x(x-1)(x-2)(x-3)$ and $\lambda_{-3}(x) = \kappa_{-3} x(x-1)(x-2)$.  The modified system would simply take $\lambda_{-3}(4) = 0$.

Let $n_u$ ($u$ for `up') be the largest $n$ for which $\kappa_n$ is a non-zero reaction rate and similarly let $n_d$ ($d$ for `down') be the largest $n$ for which $\kappa_{-n}$ is a non-zero rate constant.    
For the network \eqref{eq:98767896}, $n_u = 4$ and $n_d=3.$

\begin{theorem}    \label{thm:bd}
  Let $\{\S, \C, \Reac\}$ be a  reaction network with one species only.  Assume that all reaction vectors are of the form $\zeta_n=(-1)$ or $\zeta_n=(1)$, and assume that there is at least one of each form. Let  $\{\kappa_n\}$ be a choice of rate constants and assume, for some $V>0$, that $\{\kappa^V_n\}$ is related to $\{ \kappa_n\}$ via \eqref{eq:kappas}. 
  Then a stationary distribution, $\pi^V$, for the modified birth-death process with rates \eqref{bd} and rate constants $\kappa_n^V$ exists on the irreducible component $\Gamma=\{i| i\ge i_0\}$ if and only if either of the following holds,
\begin{itemize}\label{cond-bd}
\item[(1)] $n_d>n_u$, \quad or
\item[(2)] $n_d=n_u$ \quad and \quad $\kappa_{-n_d}>\kappa_{n_u}$,
\end{itemize}
in which case such a $\pi^V$ exists for each choice of $V>0$.  

If either of conditions (1) or (2) holds, and if $\tilde x^V \to \tilde x \in (0,\infty)$, where each $\tilde x^V \in \frac1V \Z_{\ge 0}$, then 
\begin{equation}\label{gx}
	\lim_{V\to \infty} -V^{-1}\ln(\tilde \pi^V(\tilde x^V)) = g(\tilde x)\eqdef - \int_{\tilde x_{\max}}^{\tilde x} \ln\left(\frac{ \sum_{n\ge 0} \kappa_nu^{\nu_{n}} }{  \sum_{n< 0} \kappa_nu^{\nu_{n}}  }\right) du,
\end{equation}
where $\tilde \pi^V$ is the stationary distribution for the  stochastic model scaled by $V>0$ and state space $\frac1V \Z_{\ge 0}$ (as in \eqref{eq:4309}), and $\tilde x_{\max}$ is a value of $\tilde x\in [0,\infty)$ (potentially not unique) that maximizes the integral
$$   \int_0^{\tilde x} \ln\left(\frac{ \sum_{n\ge 0} \kappa_nu^{\nu_{n}} }{  \sum_{n< 0} \kappa_nu^{\nu_{n}}  }\right) du. $$ 
Further, the function $g(\tilde x)$ of \eqref{gx}  fulfills condition (2) in Definition \ref{def:lya}; that is, $g(\tilde x)$ decreases along paths of the deterministically modeled system with rate constants $\{\kappa_n\}$.
 \end{theorem}

\begin{proof}
Since all reactions have $\zeta_n=(1)$ or $\zeta_n=(-1)$ it follows that the system is equivalent to a birth-death process with  birth and death rates \eqref{bd}. As in the discussion below \eqref{bd}, let $i_0$ be the smallest value the chain may attain. Potentially after modifying the system as detailed above, we have that  $p_i>0$ for all $i\ge i_0$ and  $q_i>0$ for all $i\ge i_0+1$.
 Hence, $\Gamma=\{i\in \Z|i\ge i_0\}$ is irreducible and the stationary distribution, if it exists, is given by (see \cite{karlin})
\begin{align*}
	\pi^V\!(x) = \frac{1}{Z^V} \prod_{i = i_0+1}^x \frac{p_{i-1}}{q_{i}} = \frac{1}{Z^V} \frac{p_{i_0}\cdots p_{x-1}}{q_{i_0+1} \cdots q_x} , \quad  x\ge i_0,
\end{align*}
where the empty product $\Pi_{i = i_0+1}^{i_0}$ is taken to be equal to 1, and
the partition function $Z^V$ satisfies
\begin{equation}\label{eq:partition}
Z^V =\sum_{x=i_0}^\infty \prod_{i = i_0+1}^x \frac{p_{i-1}}{q_{i}}.
\end{equation}
Let $\delta=n_d-n_u$. Note that for $\epsilon>0$ arbitrarily small, there exists an $m>0$ such that
\begin{equation}\label{eq:bounds}
(1+\epsilon)\frac{V^\delta}{i^\delta} \frac{ \kappa_{n_u}}{\kappa_{-n_d}}\,\ge\,	\frac{p_{i-1}}{q_i} \,\ge\,  (1-\epsilon)\frac{V^\delta}{i^\delta} \frac{ \kappa_{n_u}}{\kappa_{-n_d}} \quad \text{for }\quad i> mV,
\end{equation}
  for all $V>0$. Hence, 
\begin{equation*}
	Z^V = \Theta\left(\sum_{i=i_0}^\infty  \frac{V^{\delta i}}{(i!)^\delta}\left( \frac{\kappa_{n_u}}{\kappa_{-n_d}}\right)^{\!i}(1+\epsilon)^i  \right),
\end{equation*}
which is finite if and only if one of the two conditions (1) and (2) in the theorem is fulfilled, in which case it is finite for all $V>0$. If $\delta=0$, one should choose $\epsilon$ such that $(1+\epsilon)\kappa_{n_u}/\kappa_{-n_d}<1$. Since a stationary distribution exists if and only if $Z^V$ is finite (see \cite{karlin}), 
this concludes the first part of the theorem.

 We assume now that the stationary distribution exists, that is, that one of the two conditions (1) and (2) are fulfilled, and consider the non-equilibrium potential.
 Letting $\tilde x^V = V^{-1}x$ with $x\ge i_0$, the scaled non-equilibrium potential takes the form
\begin{align}
	-V^{-1}\ln(\tilde \pi^V\!(\tilde x^V))&= -V^{-1} \ln(\pi^V\!(V\tilde x^V))\notag\\[1ex]
	&= -V^{-1}\left[ \sum_{i=i_0+1}^{V\tilde x^V} (\ln(p_{i-1}) - \ln (q_i)) \right] +V^{-1}\ln(Z^V).
	\label{eq:98879}
\end{align}
Using the definitions of $p_i$, $q_i$ and $\lambda_n^V\!(i)$,  the  first term  in \eqref{eq:98879} becomes
\begin{equation*}
	 -V^{-1}\sum_{i=i_0+1}^{V\tilde x^V} \left[ \ln\left(\sum_{n\ge 0} \kappa_n\frac{(i-1)(i-2)\cdots (i-|\nu_{n}|)}{V^{|\nu_{n}|-1}} \right)- \ln\left(\sum_{n< 0} \kappa_n\frac{i(i-1)\cdots (i-|\nu_{n}|+1)}{V^{|\nu_{n}|-1}} \right)\right].
\end{equation*}
Noting that this is a Riemann sum approximation, we have for $\tilde x^V\to\tilde x\in (0,\infty)$,
\begin{align}\label{eq:riemann}
	-V^{-1}\sum_{i=i_0+1}^{V\tilde x^V}\left[ \ln(p_{i-1}) - \ln(q_i)\right] &\to - \int_0^{\tilde x} \ln\left(\frac{ \sum_{n\ge 0} \kappa_nu^{\nu_{n}} }{  \sum_{n< 0} \kappa_nu^{\nu_{n}}  }\right) du \eqdef g_1(\tilde x),
\end{align}
as $V \to \infty$. 
 
We next turn to the second term of \eqref{eq:98879}. First, we consider the infinite series in equation \eqref{eq:partition}.  By \eqref{eq:bounds}, for $\epsilon>0$ small enough there is an $m>0$ so that if $i>  mV$,  then
 \begin{equation}\label{eq:term}
  \frac{p_{i-1}}{q_i}\le 
  (1+\epsilon)\frac{k_{n_u}}{k_{-n_d}}\frac{1}{m}\eqdef \beta<1.
  \end{equation}
Let $m_V=\lfloor  mV\rfloor+1$. Hence, it follows that the tail of the partition function $Z^V$ fulfills
\begin{align}\label{eq:termMv}
\sum_{x=m_V}^\infty\prod_{i=i_0+1}^x \frac{p_{i-1}}{ q_i} &=\left(\prod_{i=i_0+1}^{m_V} \frac{p_{i-1}}{ q_i}\right)\sum_{x=m_V}^\infty\prod_{i=m_V+1}^x \frac{p_{i-1}}{ q_i} \nonumber \\
	&\le \left(\prod_{i=i_0+1}^{m_V} \frac{p_{i-1}}{ q_i}\right)\sum_{x=m_V}^\infty  \beta^{x-m_V} \nonumber \\
	& = \left(\prod_{i=i_0+1}^{m_V} \frac{p_{i-1}}{ q_i} \right)\frac{1}{1-\beta}.
\end{align}
Next we bound  $Z^V$ above, using \eqref{eq:termMv}, 
\begin{align}\label{eq:partition2}
	Z^V & =\sum_{x=i_0}^\infty \prod_{i = i_0+1}^x \frac{p_{i-1}}{q_{i}}
 \nonumber \\
	 & \le \sum_{x=i_0}^{m_V-1}\prod_{i = i_0+1}^x \frac{p_{i-1}}{q_{i}} +  \left(\prod_{i=i_0+1}^{m_V} \frac{p_{i-1}}{ q_i} \right)\frac{1}{1-\beta} \nonumber \\
    &=	\sum_{x=i_0}^{m_V} \exp\left(\sum_{i=i_0+1}^{x}  \left[\ln(p_{i-1}) - \ln(q_i)\right] \,-\delta_{m_V}(x)\ln(1-\beta)\right) \nonumber \\
    & \le\frac{1}{1-\beta}\sum_{x=i_0}^{m_V} \exp\left(\sum_{i=i_0+1}^{x} [ \ln(p_{i-1}) - \ln(q_i)] \right),
\end{align}
with the convention that the empty sum is zero, and 
where $\delta_{a}(x)$ is an indicator function that takes the value $1$ if $x=a$, and is zero otherwise. In the last inequality we have used that $-\delta_{a}(x)\ln(1-\beta)\le -\ln(1-\beta)$.

Consider the right side of \eqref{eq:partition2}. Let $x_V$ be the value of $x\le m_V$ for which the sum attains it maximum. Hence, we have
\begin{align}\label{eq:partition3}
	Z^V &  \le\frac{m_V}{1-\beta}\exp\left(\sum_{i=i_0+1}^{x_V}  \ln(p_{i-1}) - \ln(q_i) \right).
\end{align}
The sequence $V^{-1}x_V\in [0,V^{-1}m_V]\subseteq [0,m+1]$ has an accumulation point $\tilde x_{\max}$ in $[0,m+1]$ since the interval is compact. Using \eqref{eq:riemann} and $m_V=\lfloor mV\rfloor +1$, we obtain from \eqref{eq:partition3}
\begin{align}\label{eq:bound-down}
	\limsup_{V\to\infty} V^{-1} \ln (Z^V)  & \le \int_0^{\tilde x_{\max}}\!\! \ln\left(\frac{ \sum_{n\ge 0} \kappa_nu^{\nu_{n}} }{  \sum_{n< 0} \kappa_nu^{\nu_{n}}  }\right) du \eqdef g_0.
\end{align}
Note that $\tilde x_{\max}$ is a global maximum of the integral on the entire $[0,\infty)$ (though it might not be unique): according to \eqref{eq:term}, the terms in the inner sum  in \eqref{eq:partition2} are negative for $x> x_V$.

To get a lower bound for $Z^V$, we choose a sequence of points $\tilde x^V\in \frac{1}{V}\Z_{\ge 0}$, such that $\tilde x^V\to \tilde x_{\max}$ as $V\to\infty$. Then, with $x_V=V\tilde x^V$,
\begin{align*}
	Z^V & \ge  \frac{p_{i_0}\cdots p_{x_V-1}}{q_{i_0+1}\cdots q_{x_V}},
\end{align*}
and consequently,
\begin{align}\label{eq:bound-up}
	\liminf_{V\to\infty} V^{-1}\ln(Z^V) & \ge  \int_0^{\tilde x_{\max}}\!\! \ln\left(\frac{ \sum_{n\ge 0} \kappa_nu^{\nu_{n}} } {  \sum_{n< 0} \kappa_nu^{\nu_{n}}  }\right) du =g_0,
\end{align}
by arguing as in \eqref{eq:riemann}.
Combining \eqref{eq:bound-down} and \eqref{eq:bound-up} yields the desired result that $V^{-1}\ln(Z^V)\to g_0$ as $V\to \infty$.

Hence, we may conclude that the non-equilibrium potential converges to the function $g(\tilde x)=g_1(\tilde x)+g_0$, as stated in the theorem. To conclude the proof, we only need to confirm that $g$ fulfills condition (2) in Definition \ref{def:lya}, which we verify by differentiation,
\begin{align*}
	\frac{d}{dt} g(x(t)) &= g'(x(t)) x'(t)\\
	&= -\ln\left(\frac{ \sum_{n\ge 0} \kappa_nx^{\nu_{n}} }{  \sum_{n< 0} \kappa_nx^{\nu_{n}}  }\right) \cdot  \left( \sum_{n\ge 0} \kappa_nx^{\nu_{n}} -  \sum_{n< 0} \kappa_nx^{\nu_{n}} \right).
\end{align*}
This is  strictly negative unless
$$ \sum_{n\ge 0} \kappa_nx^{\nu_{n}} -  \sum_{n< 0} \kappa_nx^{\nu_{n}}=0,$$
in which case we are at an equilibrium.
\end{proof}

For this particular class of systems we have
\begin{align*}
	\dot x &= \sum_{n\ge 0} \kappa_nx^{\nu_{n}} -  \sum_{n< 0} \kappa_nx^{\nu_{n}},
\end{align*}
so that the ratio in equation \eqref{gx} is simply  the ratio of the two terms in the equation above.
The local minima and maxima of $g(\tilde x)$ are therefore the equilibrium points of the deterministically modeled system. Further, by inspection, it can be seen that $g(\tilde x_{\text{max}})=0$  and $g(\tilde x)\to\infty$ as $\tilde x\to\infty$. If none of the extrema of $g(\tilde x)$ are plateaus, then it follows that asymptotically stable and unstable equilibria must alternate and that the largest equilibrium point is asymptotically stable (Definition \ref{def:lya}). Around each of the stable equilibria the function $g(\tilde x)$ is a Lyapunov function.

\begin{example}
\label{example:10}
 Consider the following network which has three equilibria (for appropriate choice of rate constants), two of which may be stable,
\begin{align*}
	\emptyset \overset{\kappa_0}{\underset{\kappa_{-1}}{\rightleftarrows}} X, \qquad 2X \overset{\kappa_2}{\underset{\kappa_{-3}}{\rightleftarrows}} 3X.
\end{align*}
The deterministic model satisfies
\begin{align*}
	\dot x = \kappa_0 - \kappa_{-1}x + \kappa_2 x^2 - \kappa_{-3}x^3.
\end{align*}
We have $n_{\rm{u}}=2$ and $n_{\rm{d}}=3$  such that condition (1) of Theorem \ref{thm:bd} is fulfilled. Hence, the non-equilibrium potential converges to the  function 
\begin{align}\label{ex1g}
g(\tilde x) &=- \int_{\tilde x_{\max}}^{\tilde x} \ln\left(\frac{\kappa_0 + \kappa_2x^2}{\kappa_{-1}x + \kappa_{-3}x^3}\right) dx.
\end{align}
The stationary distribution of the stochastically modeled system can be obtained in closed form \cite{Gardiner1985},
$$\pi^V\!(x)=\pi^V\!(0) \prod_{i=1}^x \frac{B[(i-1)(i-2)+P]}{i(i-1)(i-2)+Ri},$$
where 
$$B=\frac{\kappa_2}{\kappa_{-3}}, \quad R=\frac{\kappa_{-1}}{\kappa_{-3}},\quad \text{and}\quad P=\frac{\kappa_0}{\kappa_2}.$$
 If $P=R$, then the distribution is Poisson with parameter $B$ and, in fact, the system is complex balanced. In this case, $\tilde x_{\max}=\kappa_2/\kappa_{-3}$ and  the Lyapunov function \eqref{ex1g} reduces to
\begin{align*}
g(\tilde x) &=  \tilde x\ln(\tilde x)-\tilde x- \tilde x\ln\left(\frac{\kappa_{2}}{\kappa_{-3}}\right) +\frac{\kappa_{2}}{\kappa_{-3}},
\end{align*}
  in agreement with Theorem \ref{comp-main}.

For a concrete example that is not complex balanced, consider the model with rate constants  $\kappa_0 = 6, \kappa_{-1} = 11, \kappa_2 = 6, \kappa_{-3} = 1$. In this case 
\begin{align*}
	\dot x= 6 - 11 x + 6x^2 - x^3 = -(x-1)(x-2)(x-3),
\end{align*}
and there are two asymptotically stable equilibria at $c=1,3$ and one unstable at $c=2$.  Hence, the function $g(\tilde x)$ is a Lyapunov function locally around $\tilde x=1,3$, and takes the form
 \begin{align}\begin{split}
 g(\tilde x) &= \tilde x \left( \ln \left( \frac{\tilde x(\tilde x^2+11)}{\tilde x^2+1}\right) - \ln(6) - 1\right) + 2\sqrt{11} \arctan\left( \frac{\tilde x}{\sqrt{11}}\right) - 2\arctan \left( \tilde x \right) \\
 &\hspace{.2in} -2\sqrt{11}\arctan\left(\frac1{\sqrt{11}}\right) +1+\frac12\pi,
 \label{eq:89679786}
 \end{split}
 \end{align}
where, for this example, $\tilde x_{\text{max}} = 1$.
In Figure \ref{fig:convergence}, we demonstrate the convergence of the scaled non-equilibrium potential, $-\frac1V \ln(\tilde \pi^V(\tilde x^V))$, of the scaled process to $g(\tilde x)$ of \eqref{eq:89679786}.
 
  \begin{figure} 
 \begin{center}
 \includegraphics[width=6in]{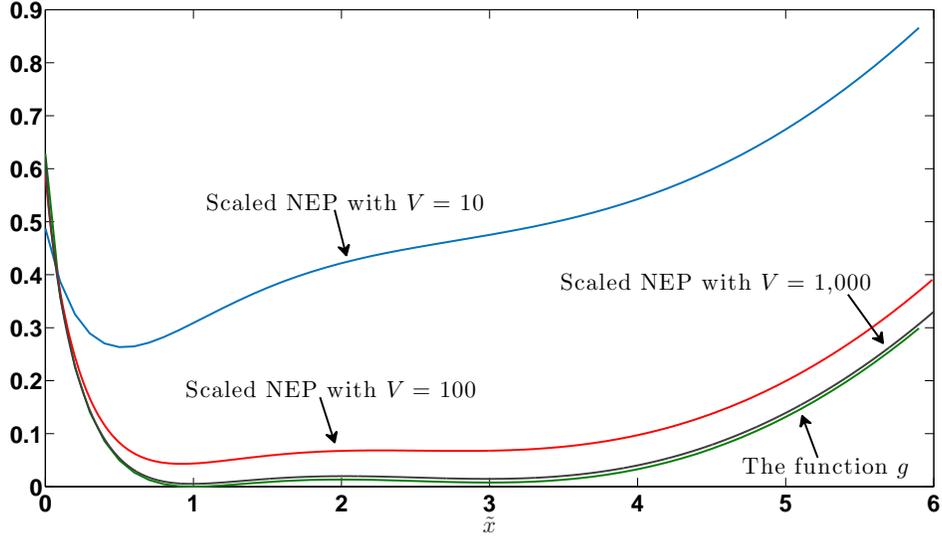}
 \end{center}
 \caption{Plots of the scaled non-equilibrium potential (NEP), $-\frac1V \ln(\tilde \pi^V(\tilde x^V))$, of the scaled birth-death process of Example \ref{example:10} are given for $V \in \{10,10^2,10^3\}$, as is the function $g(\tilde x)$ of \eqref{eq:89679786}.}
 \label{fig:convergence}
 \end{figure}
\hfill $\square$
\end{example}

\begin{example}
 Consider the  reaction network
\begin{align*}
	X \overset{k_{-1}}{\rightarrow} \emptyset, \qquad X \overset{k_1}{\rightarrow} 2X,
\end{align*}
which is equivalent to a linear birth-death process with absorbing state 0. This model has $n_{\rm{u}}=n_{\rm{u}}=1$,  and so for a stationary distribution to exist the second condition of Theorem  \ref{thm:bd} must hold. If we put  the death rate $\lambda_{-1}(1)$ to 0 and assume $\kappa_{-1}>\kappa_1$, then  condition (2) is fulfilled and
\begin{align}\label{exam}
g(\tilde x) &=- \int_0^{\tilde x} \ln\left(\frac{\kappa_1 x}{\kappa_{-1}x}\right) dx= -\tilde x \ln\left(\frac{\kappa_1 }{\kappa_{-1}}\right)
\end{align}
is a Lyapunov function.
 In fact, the stationary distribution of the modified system is proportional to
$$ \pi^V\!(x)\propto\left(\frac{\kappa_{1}}{\kappa_{-1}}\right)^{\!x-1}\frac{1}{x},$$
which is independent of $V$. It follows that for $\tilde x^V\to\tilde x$,
\begin{align*}
-\frac{1}{V}\ln(\tilde \pi^V\!(\tilde x^V)) &\approx -\left(\tilde x^V-\frac{1}{V}\right)\ln\left(\frac{\kappa_{1}}{\kappa_{-1}}\right)+\frac{1}{V}\ln(\tilde x^V)+\frac{1}{V}\ln(V) \\
	& \to - \tilde x \ln\left(\frac{\kappa_1 }{\kappa_{-1}}\right),
\end{align*}
in agreement with \eqref{exam}. In this particular case the deterministic system converges to zero -- the absorbing state of the stochastic system -- though this correspondence will not hold in general for systems with an absorbing state.
\hfill $\square$
\end{example}

\subsection{Other examples}

\begin{example}\label{endo1}
Consider the  reaction network,
\begin{equation*}
\emptyset\stackrel{\kappa_1}{\to} X,\qquad 2X\stackrel{\kappa_2}{\to} \emptyset.
\end{equation*}
The network is not complex balanced, nor is it a birth-death process, hence the theory developed in the previous sections is not applicable. The stationary distribution with scaled rate constants as in \eqref{eq:kappas} can be given in explicit form \cite{Engblom},
\begin{equation*}
\pi(x)=\frac{1}{\sqrt{2}I_1(2\sqrt{2}aV)}\frac{(aV)^x}{x!}I_{x-1}(2aV),\quad x\in\Z_{\ge 0},\quad \text{and}\quad a=\sqrt{\frac{\kappa_1}{\kappa_2}},
\end{equation*}
where $I_n(z)$ is the modified Bessel function of the $n$th kind. To evaluate the non-equilibrium potential we need two asymptotic results for the modified Bessel functions \cite{gradshteyn}:
\begin{align*}
I_1(z) &\propto \frac{1}{\sqrt{2\pi z}}e^{z}, \quad \text{for large }  z, \\
I_n(nz) &\propto \frac{1}{\sqrt{2\pi n}}\frac{e^{\eta n}}{(1+z^2)^{1/4}}\left( \,1+\sum_{k=1}^\infty \frac{u_k(t)}{n^k}\,\right), \quad \text{for large }  n
\end{align*}
where 
$$\eta=\sqrt{1+z^2}+\ln\left(\frac{z}{1+\sqrt{1+z^2}}\right),\qquad t =\frac{1}{\sqrt{1+z^2}},$$
and  $u_k(t)$, $k\ge 1$,  are functions of $t$. Note that the sum involving $u_k(t)$ decreases proportionally to $n^{-1}u_1(t)$ as $n$ gets large (the other terms vanish faster than $\frac{1}{n}$).

After some cumbersome calculations using the asymptotic relationships for the modified Bessel function, we obtain that  the non-equilibrium potential satisfies
\begin{equation*}
 -\frac{1}{V}\ln(\tilde\pi^V\!(\tilde x^V)) \to g(\tilde x), \quad\text{for }\quad  \tilde x^V\to \tilde x\quad\text{ as}\quad V\to\infty,
 \end{equation*}
 where $g(\tilde x)$ is defined by
 \begin{equation*}
 g(\tilde x) = 2\sqrt{2}a -2\tilde x\ln(a)+\tilde x\ln(\tilde x)-\tilde x(1+\ln(2))-\sqrt{\tilde x^2+4a^2}+\tilde x\ln(\tilde x+\sqrt{\tilde x^2+4a^2}).
 \end{equation*}
Another straightforward, but likewise cumbersome, calculation, shows that $g(\tilde x)$ in fact fulfils condition (2) in Definition \ref{def:lya}. By differentiation twice with respect to $x$, we find that $g''(\tilde x)>0$, hence $g(\tilde x)$ is  a Lyapunov function.
\hfill $\square$
\end{example}

\begin{example}\label{endo2}
 As a last example consider the  reaction network:

\begin{equation*}
X\stackrel{\kappa_1}{\to} \emptyset,\qquad \emptyset\stackrel{\kappa_2}{\to} 2X.
\end{equation*}
It is not weakly reversible, hence not complex balanced for any choice of rate constants. It is not a birth-death process either, as two molecules are created at each ``birth" event. It is similar to Example \ref{endo1}, but with the reactions going in the opposite direction.

Let the rate constants $\{\kappa_k\}$ be given and let the scaled rates $\{\kappa^V_k\}$ be given accordingly.  The deterministically modeled system takes the form
\begin{align}\label{det-eqn}
\dot{x}=2\kappa_2-\kappa_1 x
\end{align}
such that there is a unique equilibrium at $c=\frac{2\kappa_2}{\kappa_1}$. Let $a\eqdef \frac{\kappa_2}{2\kappa_1}$ so that  $c=4a$.
The stationary distribution  exists for all reaction rates and is most easily characterized in the following way (see Supporting Information):
$$N=N_{1}+2N_{2},\qquad N_{1}\sim \po(2a V), \quad \text{and}\quad N_{2}\sim \po\left(a V\right),$$
where $N_1$ and $N_2$ are two  independent Poisson random variables with intensities $2a V$ and $a V$, respectively. Hence, the stationary distribution can be written as
\begin{align*}
\pi(x) 	&= e^{-3Va} \sum_{k,m\colon x=k+2m}   \frac{(2Va)^k}{k!}\frac{(Va)^{m}}{ m!}.    
\end{align*}

  In the Supporting Information it is shown that the  limit of the non-equilibrium potential exists as $V\to \infty$ with  $\tilde x^V\to\tilde x$:
$$\lim_{V\to\infty}-\frac{1}{V}\ln(\tilde\pi^V\!(\tilde x^V) )= g(\tilde x),$$
where
\begin{align*}
g(\tilde x) & = \int_0^{\tilde x} \ln\left(\sqrt{1+\frac{2 x}{a}}-1\right) dx -\ln(2)\,\tilde x
\end{align*}
(the integral can be solved explicitly, see Supporting Information).
The first derivative of $g$ fulfils
$$g'( x)>0\quad \text{if and only if} \quad 4a< x,$$
and zero if and only if $4a= x$. Comparing with \eqref{det-eqn} yields 
$$g'(x) \dot{x}\le 0 \quad \text{for all}\quad x>0,$$
and equality only if $4a=x$. The second derivative of $g$  is positive for all $x$. Hence, $g(x)$ is a Lyapunov function.
\end{example}

\section{Discussion}

We have demonstrated a relationship between the \textit{stochastic} models for reaction systems and an important Lyapunov function for the corresponding \textit{deterministic} models.  In particular, we showed that this relationship holds for the class of complex balanced systems, which contains the class of detailed balanced systems that have been well studied in both the physics and probability literature \cite{Whittle86}.  Further, we showed the  correspondence  holds for a wider class of models including those birth and death systems that can be modeled via reaction systems.  It remains open just how wide the class of models satisfying this relationship is.

\vspace{.2in}
\noindent \textbf{Acknowledgements.}
We thank the American Institute of Mathematics for hosting a workshop at which this research was initiated.    Anderson was supported by NSF grants DMS-1009275 and DMS-1318832 and Army Research Office grant W911NF-14-1-0401. Craciun was supported by NSF grant DMS1412643 and NIH grant R01GM086881.    Wiuf was supported by the Lundbeck Foundation (Denmark), the Carlsberg Foundation (Denmark), Collstrups Fond (Denmark), and the Danish Research Council.  Part of this work was carried out while Wiuf visited the Isaac Newton Institute in 2014. 

\bibliographystyle{amsplain} 
\bibliography{LyapunovScaling}

\end{document}